\documentclass{amsart}

\usepackage{graphicx}
\usepackage{fancyhdr}
\usepackage{pst-node}
\usepackage{tikz-cd} 
\usepackage{hyperref}
\theoremstyle{remark}
\newcommand{\ric}{\mathrm{Ric}}

\newtheorem{thm}{Theorem}[section]
\newtheorem{cor}[thm]{Corollary}
\newtheorem{lem}[thm]{Lemma}

\newtheorem{rmk}[thm]{Remark}
\newtheorem{prop}[thm]{Proposition}

\numberwithin{equation}{section}

\begin{document}
\title{Asymptotic curvature estimate for steady solitons}
\author{Daoyuan Han}
\address{Department of Mathematics, Lehigh University, Bethlehem, Pennsylvania 18015}
\curraddr{Department of Mathematics, Lehigh University, Bethlehem, Pennsylvania 18015}
\email{dah517@lehigh.edu}





\begin{abstract}
In this note, we shall investigate the asymptotic curvature estimate on steady Ricci solitons.
\end{abstract}
\maketitle
\section{Main theorems}
\begin{thm}\label{main}
Let $(M,g,f)$ be a complete $n$-dimensional $\kappa$-noncollapsed steady soliton with nonnegative sectional curvature and positive Ricci curvature. Assuming scalar curvature $R(x)\rightarrow 0$ as $d(x)\rightarrow \infty$ uniformly with respect to distance, then 
\begin{equation}
    \liminf_{d(x)\rightarrow \infty} R(x)d(x)^{\alpha}=0,
\end{equation}
for all $\alpha\in (0,1)$, where $d(x)$ is the distance from $x\in M$ to a fixed point $x_0\in M$ and $C$ is a constant independent of $x$. 
\end{thm}



Under the additional assumption 
\begin{equation}\label{cond0}
   \ric(\nabla f, \nabla f)\geq \frac{C}{d(x)^2} \text{ as } d(x)\rightarrow \infty
\end{equation}
for some constant $C>0$
and
\begin{equation}\label{cond1}
\lim_{d(x)\rightarrow \infty} R(x)d(x)^\alpha \text{ exists for some }\alpha\in [4/5,1),
\end{equation}
we can prove the following result. 
\begin{thm}\label{main2}
Let $(M,g,f)$ be a complete $n$-dimensional $\kappa$-noncollapsed steady soliton with nonnegative sectional curvature and positive Ricci curvature. Assuming scalar curvature $R(x)\rightarrow 0$ as $d(x)\rightarrow \infty$ uniformly with respect to distance, (\ref{cond0}) and (\ref{cond1})
then $R(x)$ must decay at least linearly with respect to $d(x)$, namely 
\begin{equation}
    R(x)\leq \frac{C}{d(x)}, 
\end{equation}
where $d(x)$ is the distance from $x\in M$ to a fixed point $x_0\in M$ and $C$ is a constant independent of $x$. 
\end{thm}
\begin{rmk}
By Proposition \ref{main}, the condition (\ref{cond1}) implies 
\[
\lim_{d(x)\rightarrow \infty} R(x)d(x)^\alpha=0 \text{ exists for some }\alpha\in [4/5,1). 
\]
The condition (\ref{cond1}) can be viewed as an asymptotic control of the curvature. 
\end{rmk}
\begin{rmk}
Note that by the result in \cite{jiaping}, under the non-collapsing condition, the scalar curvature can't decay faster than linear, by the uniform decaying condition, we can write 
\begin{equation}\label{2}
   R(x)\geq \frac{C'}{d(x)}
\end{equation}
for some constant $C'>0$, independent of $x$.  
\end{rmk}





\section{Preliminary}
$(M,g,f)$ is called a gradient Ricci soliton if $\ric$ of $g$ satisfies 
\begin{equation}
    \ric=\nabla^2 f.
\end{equation}
and $f$ is a smooth function on $M$. In this note, we always assume the steady Ricci soliton has positive Ricci curvature. We also assume that $R$ attains a maximum at some point $p$ on $M$, so by the identity (after normalization)
\begin{equation}\label{soleqn}
    |\nabla f|^2+R=1.
\end{equation}
we know $p$ is a critical point of $f$ and by the assumption that $\ric>0$, $p$ is the unique critical point of $f$. 

\section{Proof of Theorem \ref{main}}
First, we list a few well-known results for steady solitons valid in higher dimensions. 


We also need the following result on the lower bound of volume growth of steady soliton in \cite{sesum}.
\begin{thm}\label{vollb}\cite{sesum}
If $(M,g)$ is a complete gradient steady Ricci soliton there exists uniform constant $C$ so that for any $r>r_0$ 
\begin{equation}
    Vol(B_p(r))\geq Cr
\end{equation}
\end{thm}
By the equivalence of $f(x)$ and distance function in \cite{caochen}, we have, under the same conditions as in the Theorem \ref{vollb} above, the following volume lower bound on the sub-level set, $\Sigma_{\leq r}:=\{x\in M| f(x)\leq r\}$.

\begin{cor}
\begin{equation}\label{volumegrowth}
    Vol(\Sigma_{\leq r})\geq Cr
\end{equation}
\end{cor}

\begin{proof}
We use the equivalence of potential function and distance function from \cite{caochen}. Let $C_1, C_2$ be constants s.t. 
\[
C_1 d(x)\leq f(x)\leq C_2 d(x)
\]
and then we have the inclusion $B_p(r)\subset \Sigma_{\leq C_2 r}$, which implies that 
\[
Vol(\Sigma_{\leq C_2 r})\geq Vol(B_p(r))\geq Cr
\]
\end{proof}

\begin{prop}\label{curvint}
Let $(M,g,f)$ be a complete $n$-dimensional $\kappa$-noncollapsed on all scales steady soliton with  $\ric> 0$, and with scalar curvature $R$ decaying like $r^{-\alpha}$, $\alpha\in (0,1)$ when $r$ is sufficiently large, namely
\begin{equation}\label{scalardecay}
    C_1r^{-\alpha}\leq R\leq C_2r^{-\alpha},
\end{equation}
then we have 
\begin{equation}\label{curvint2}
\int_{\Sigma_{\leq r}} R\,d\mu\geq C r^{1-\alpha} \text{ for } r \text{ sufficiently large},
\end{equation}
where $C$ denotes a positive constant depending on $\alpha$. 
\end{prop}
\begin{proof}
Using the result in \cite{sesum} on the lower bound of volume growth and Corollary \ref{volumegrowth}, let $r_1$ be such that Corollary \ref{volumegrowth} and (\ref{scalardecay}) hold,  then we know, by coarea formula,  
\[
\int_{\Sigma_{\leq r}} R\,dV-\int_{\Sigma_{\leq r_1}} R\,dV= \int_{\{x|r_1 \leq f(x)\leq r\}}R\,dV\geq \int_{r_1}^r \int_{\Sigma_s}R\,d\mu ds\geq C r^{1-\alpha},
\]
for $r$ sufficiently large. The claim follows from the above. 
\end{proof}
The left hand side of the above is by the computation in \cite{sesum}
\begin{equation}\label{vollevel}
    \int_{\Sigma_{\leq r}} R\,dV= \int_{\Sigma_r} \Delta f \,dV\leq \int_{\Sigma_r} |\nabla f|\,d\mu\leq Vol(\Sigma_r)
\end{equation}
From the above estimate, we have the following estimate. 
\begin{cor}\label{sublev}
Under the same conditions as in Proposition \ref{curvint}, we have 
\begin{equation}
    Vol(\Sigma_r)\geq C r^{1-\alpha}, \text{ for }r \text{ sufficiently large}.
\end{equation}
where $C>0$ is a constant depending on $\alpha$.  
\end{cor}
By the above corollary, we have the following lower bound of $Vol(\Sigma_{\leq r})$,
\begin{prop}\label{level}
Under the same conditions as in Proposition \ref{curvint}, we have 
\begin{equation}\label{vol}
    Vol(\Sigma_{\leq r})\geq C r^{2-\alpha}, \text{ for }r \text{ sufficiently large}.
\end{equation}
\end{prop}
\begin{proof}
By the coarea formula, we have 
\[
Vol(\Sigma_{\leq r})\geq \int_{0}^r\int_{\Sigma_s}\,d\mu=\int_{r_0}^r\int_{\Sigma_s}\,d\mu ds+\int_0^{r_0}\int_{\Sigma_s}\,d\mu ds,
\]
where $r_0$ is chosen so that Corollary \ref{sublev} holds. By choosing $r$ sufficiently large, we have 
\[
Vol(\Sigma_{\leq r})\geq Cr^{2-\alpha}
\]
where $C>0$ and depends on $\alpha, r_0$. 
\end{proof}
We will use the above results to prove Proposition \ref{main}. The proof is by contradiction and iteration of the above process. 
\begin{proof}
Suppose by contradiction that $(M,g)$ has the limit 
\[
\liminf_{d(x)\rightarrow \infty} R(x)d(x)^\alpha=C>0. 
\]
We apply the volume estimate (\ref{vol}) back to the formula (\ref{curvint}). Choose $r_1>0$ s.t. Proposition \ref{level} holds, we have 
\begin{equation}\label{1}
        \int_{\Sigma_{\leq r}} R\,d\mu \geq C r^{2-2\alpha},
\end{equation}
for $r\geq r_1$. 

By the same proof as in (\ref{vollevel}), we get
\begin{equation}\label{4}
     Vol(\Sigma_{r})\geq C r^{2-2\alpha}. 
\end{equation}
Now we iterate the process (\ref{1}-\ref{4}), then after $k$ iterations, we have 
\begin{equation}
    Vol(\Sigma_{r})\geq C r^{k(1-\alpha)}, \text{ for } r \text{ sufficiently large}.
\end{equation}
So when $k$ is sufficiently large, we get
\begin{equation}
Vol(\Sigma_{r})\geq C r^n, \text{ for } r \text{ sufficiently large}.
\end{equation}
Using the equivalence of potential function and distance function from \cite{caochen}. Let $C_1', C_2'$ be constants s.t. 
\[
C_1' f(x)\leq d(x)\leq C_2' f(x)
\]
and then we have the inclusion $\Sigma_{\leq C_1' r}\subset B_p(r)$, which implies that 
\[
 Vol(B_p(r))\geq Vol(\Sigma_{\leq C_1' r})\geq Cr^n,
\]
which contradicts the result that asymptotic volume ratio equals 0 for $\kappa$-noncollapsed steady solitons. 

\end{proof}

\section{Proof of Theorem \ref{main2}}
The proof follows similar ideas of choosing barrier functions as in the work of Chan \cite{chan}. We first show that the function $e^{-1/R^2}$ is subharmonic with respect to the weighted Laplacian associated to the function $f$. 

\begin{lem}\label{subharm}
Let $(M,g,f)$ be a complete gradient steady soliton satisfying the conditions in Theorem \ref{main2}. 
Then the function $e^{-1/R^2}$ satisfies
\begin{equation}
    \Delta_{f}(e^{-1/R^2})\geq 0 \text{ for } d(x) \text{ sufficiently large}.
\end{equation}
\end{lem}
\begin{proof}
The proof follows from direct computation. Firstly, 
\begin{align}\label{est1}
\begin{split}
\Delta_f(e^{-1/R^2})&=e^{-1/R^2}(4R^{-6}-6R^{-4})|\nabla R|^2+2e^{-1/R^2}\cdot R^{-3}\cdot \Delta_f R\\
&\geq e^{-1/R^2}(4R^{-6}-6R^{-4})|\nabla R|^2-2e^{-1/R^2}\cdot R^{-1}\\
&=e^{-1/R^2}R^{-6}\bigg[(4-6R^2)|\nabla R|^2-2R^5\bigg]\\
&= e^{-1/R^2}R^{-6}\bigg[4(4-6R^2)|\ric(\nabla f)|^2-2R^5\bigg]\\
&\geq e^{-1/R^2}R^{-6}\bigg[4(4-6R^2)|\ric(\nabla f, \nabla f)|^2-2R^5\bigg]\\
&\geq e^{-1/R^2}R^{-6}\bigg[4(4-6R^2) f(x)^{-4}-c f(x)^{-5\alpha}\bigg],
\end{split}
\end{align}
where in the last inequality we use Theorem  \ref{main}, the condition (\ref{cond0}) and the condition (\ref{cond1}). By choosing $\alpha\in [4/5,1)$, we know that 
\[
\Delta_f(e^{-1/R^2})\geq 0,  
\]
when $f(x)$ is sufficiently large. 

\end{proof}

Next we show that $e^{-f^2(x)}$ is supersolution of $\Delta_f$. 
\begin{lem}\label{supharm}
Let $(M,g,f)$ be a complete gradient steady soliton satisfying the conditions in Theorem \ref{main2}. Then the function $e^{-f^2}$ satisfies 
\begin{equation}
    \Delta_{f}(e^{-f^2})\leq 0.
\end{equation}
\end{lem}
\begin{proof}
The proof follows from direct computation. 
\begin{align*}
    \Delta_f(e^{-f^2})&=(4e^{-f^2}\cdot f^2-2e^{-f^2})|\nabla f|^2-2e^{-f^2}\cdot f\cdot \Delta_f f\\
    &=2e^{-f^2}\bigg[(2f^2-1)|\nabla f|^2-f\bigg]\\
    &\leq 0,
\end{align*}
for $d(x)$ sufficiently large since $f(x)$ is equivalent to $d(x)$ when $d(x)$ is sufficiently large. And 
in the second line we use the fact that $\Delta_f f=1$ with respect to the normalization in (\ref{soleqn}). 
\end{proof}

Now we use the above lemmas to prove Theorem \ref{main2}. 
\begin{proof}
Using Lemma \ref{subharm} and Lemma \ref{supharm}, we can choose a large ball $B(R_0)$ such that, on $M\setminus B(R_0)$, 
\begin{equation}
    \Delta_f(e^{-1/R^2})\geq 0
\end{equation}
and 
\begin{equation}
    \Delta_f(e^{-f^2})\leq 0
\end{equation}
Pick a large $b>0$ such that on $\partial B(R_0)$, 
\begin{equation}
    S(x):=e^{-1/R^2}-be^{-f^2}<0. 
\end{equation}
By the fact that $\lim_{f(x)\rightarrow \infty}S(x)\rightarrow 0$ and 
\begin{equation}
    \Delta_f S(x)\geq 0, 
\end{equation}
we use maximum principle to get 
\begin{equation}\label{maineq}
    e^{-1/R^2}-be^{-f^2}\leq 0, 
\end{equation}
on $M\setminus B(R_0)$, which implies 
\[
R(x)\leq \frac{C}{f(x)}, 
\]
for $C>0$ on $M\setminus B(R_0)$. 
Using the equivalence of $f(x)$ and $d(x)$ in \cite{caochen}, 
\[
R(x)\leq \frac{C'}{d(x)}. 
\]
\end{proof}

\end{document}